\documentclass[BCOR8mm,DIV14]{scrartcl}
\usepackage{amsmath, amsthm, amscd, amsfonts, amssymb, graphicx, color}
\usepackage[all,cmtip,arrow,matrix,curve]{xy}
\newtheorem{theorem}{Theorem}[section]
\newtheorem{lemma}[theorem]{Lemma}
\newtheorem{proposition}[theorem]{Proposition}

\newtheorem{definition}[theorem]{Definition}

\newtheorem{remark}[theorem]{Remark}
\numberwithin{equation}{section}
\begin{document}   
\title{Rings and Fields from Semigroups}   
\author{VOLKER W. TH\"UREY        
     \\ Bremen,  Germany    \thanks{ 49 (0)421  591777, volker@thuerey.de     }   }
 \maketitle
 \begin{abstract}   
     We introduce a ring and a field, generated by a semigroup, and we investigate some of their properties. 
 \end{abstract}
   \section{  }      
    
  	 We start with a semigroup $ ( \mathbb H, \cdot) $, where `$ \cdot $'  means the operation 
  	 in the semigroup $ \mathbb H $.
  \begin{definition}   \label{die erste definition}    
     Let \ $ ( \mathbb H,  \cdot ) $ \ be a semigroup, \,  $\mathbb H \neq \emptyset $.  
     We define the set \  $ sums({\mathbb H}) $.  
        \begin{align*} 
         sums({\mathbb H})  :=
    & \,  \{ \left\langle  \pm_1 x_1 \pm_2 x_2 \pm_3 x_3 \, .... \, \pm_{N-1}  x_{N-1} \pm_N x_N \, \right\rangle \              \vert \ N \in  {\mathbb N_0}, \ \pm_i \in \{ +,- \} \text{ (the sum is only formal)}, \  \\
    &     \text{ this means that the symbol} \ \pm_i \ \text{ \textrm  either is the sign `+' }, \, 
          \text{\textrm or } \, \pm_i  \ \text{ is the sign  `$-$'},                                  \\ 
    &     x_1, x_2, x_3, \, \ldots \, , \, x_{N-1}, x_N   \ \  
          \text{ are elements (not necessarily distinct) of the semigroup } {\mathbb H} \ \} \ .    \\
      \end{align*}
  \end{definition}    
            The set $ sums ({\mathbb H}) $ contains at least three elements
            $ \left\langle \right\rangle, \left\langle +h\right\rangle, \left\langle -h\right\rangle 
            \, \, \text{for} \  h \in \mathbb H $.  \
   It even has infinite many elements, it contains the subset 
   $ \{ \left\langle +h\right\rangle, \left\langle +h +h \right\rangle , 
                         \left\langle +h +h +h \right\rangle   , \, \ldots  \  \}   $.  
                 
   We call the number $ N  $ the { \textit {length}} of $ x $, and we call the symbols 
      $ \pm_{l} \, x_{l} $ the { \textit {entries }} of $x$, for
      $x =  \left\langle  \pm_1 x_1 \pm_2 x_2 \pm_3 x_3 \, \ldots \, \pm_{N-1}  x_{N-1} \pm_N x_N \, \right\rangle 
      \   \in sums ({\mathbb H}) $.       
                                       
      We define an equivalence relation on  $ sums ({\mathbb H}) $. Let $ x,y $ be elements of 
      $   sums({\mathbb H})  $.  We introduce the relation `$\approx_1$' to create a zero in the ring 
      $ {\mathbf{R}( \mathbb H)} $ we will construct below.                      \\
      We say that the elements $x $  and $ y $ are equivalent if \  $ x $ has length $ N $, $ y $ has length 
      $ N - 2 $, and $y$ has the same entries as $x$ at the same positions, except $ y $ lacks 
      a pair of entries $ \{+h, -h \} , h \in {\mathbb H} $. \ In symbols we write $ x \approx_1 y $. We get  \
                     $ \left\langle +h -h \right\rangle \approx_1 \left\langle \right\rangle $ and
                     $ \left\langle -h +h \right\rangle \approx_1 \left\langle \right\rangle $
      for each $ h \in {\mathbb H} $.   
      
 We add a further relation `$\approx_2$' to enforce the commutativity of the addition 
      `$ +_{\mathbb H}$' in the ring $ {\mathbf{R}(\mathbb H}) $.   
  
  We say that two elements $ x $ and $ z $ fulfill the relation `$\approx_2$' if 
      \begin{align} \label{permutation}
      x \, = \,  \left\langle  \sum_{i=1}^{N} \pm_{i} \, x_i \, \right\rangle \ \text{ and }\ 
          z :=   \left\langle  \sum_{i=1}^{N} \pm_{\nu(i)} \, x_{\nu(i)} \, \right\rangle \ \text{where} \ 
             \ \nu \  \text{ is any permutation of the set } \ \{ 1, 2, \ldots , N \} \ . 
      \end{align}       
  Let `$ \cong $' be the smallest equivalence relation that includes the relations `$\approx_1$' and `$\approx_2$'. 
      We name the equivalent classes by $[ . ]_{/ \cong} $, i.e. here we get $[ x ]_{/ \cong} =  [ y ]_{/ \cong}$ and 
      $[ x ]_{/ \cong} =  [ z ]_{/ \cong}$. \  
      We get a set of equivalent classes \ $sums({\mathbb H)}_{/ \cong}$.  
      Every equivalent class  $ [ x ]_{/ \cong} $ has an infinite number of elements. 
      Every equivalent class  contains an element $ p $ of minimal length, this means that  $ p $ contains no pair
      $ \{ +h, -h \} $ of entries. 
      In the case of an element  $[ x ]_{/ \cong} \in  sums({\mathbb H)}_{/ \cong} $  with an element $ p $
      of minimal length $ 0 , \, \text{ i.e. \ } p = \left\langle \right\rangle \in  sums({\mathbb H)} $,   
      we will write  $ [p]_{/ \cong} =: 0_{\mathbf{R}(\mathbb H)} $.         \\                                                                                         
   In the following we will not distinguish between the equivalent class $ [ x ]_{/ \cong} $ and its
      representative $ x $.  This can be made since all the operations which are constructed in the following are
      independent  from the picked representatives of the equivalent classes. 
                                                 
   We define the sets \  ${\mathbf{R}( \mathbb H)}$ and $ {\mathrm{Quot}} $.  
   \begin{definition}   \label{die zweite definition}    
     $ {\mathbf{R}(\mathbb H)} \ := \ sums({\mathbb H)}_{/ \cong} \\
         {\mathrm{Quot}} \ := \  \left\{ \, p / q  \ \vert \ p,q \in  {\mathbf{R}(\mathbb H)} 
            \ | \ \ \ q \neq  0_{\mathbf{R}(\mathbb H)}                                                   \right\} $
        \end{definition}  
       We have that $ {\mathbf{R}( \mathbb H)} $ contains at least some elements, one is 
       $ 0_{\mathbf{R}(\mathbb H)} $. \\       
       Note that the quotients in $ {\mathrm{Quot}}$  are only formal. 
       We will write $ 0_{\mathrm{Quot} } $ instead of  $ 0_{\mathbf{R}(\mathbb H)} / q $ \
       for all $ q \in sums({\mathbb H}), q \neq 0_{\mathbf{R}(\mathbb H)} $. Note that we will construct a 
       field $ { \mathbf{F}( \mathbb H)} $, and we assume that in this case the semigroup  $ ( \mathbb H, \cdot) $  
       is commutative.  \\
       We find informations on rings, fields and semigroups in \cite{Bosch} and \cite{Hungerford}.   \\
   \begin{proposition}   \label{die erste proposition}
             The tripel \, $ ( {\mathbf{R}(\mathbb H)} , +_{\mathbb H}, *) $ \, is a ring, where 
             $ +_{\mathbb H}, * :    {\mathbf{R}(\mathbb H)} \times {\mathbf{R}(\mathbb H)} \longrightarrow  
            {\mathbf{R}(\mathbb H)} $. 
   \end{proposition}
   Furthermore, we will define two operations \ $ +_{\mathrm{Quot}}, \cdot_{\mathrm{Quot}}: 
   {\mathrm{Quot}} \times {\mathrm{Quot}} \longrightarrow {\mathrm{Quot}} $. 
   \ Note that we omit mostly the symbol `$ [ . ]_{/ \cong} $' of equivalent
   classes in the following. 
  \begin{proof}                                                             
   We  define \ $ x *  0_{\mathbf{R}(\mathbb H)} :=  0_{\mathbf{R}(\mathbb H)} =: 0_{\mathbf{R}(\mathbb H)} * x $
   for each $  x \in \mathbf{R}(\mathbb H) $. \
       For positive natural numbers   $ N, K \text{ for }  \  
       x = \left\langle  \sum_{i=1}^{N} \pm_{i,1} \, x_{i} \, \right\rangle \ , \ 
       y =  \left\langle  \sum_{j=1}^{K} \pm_{j,2} \, y_j \, \right\rangle ,
       \ x , y , p , q , r , s \ \in  sums({\mathbb H)} $
       \ we define the multiplications for  \ $ q,s \neq 0_{\mathbf{R}(\mathbb H)} $   
     \begin{align*}  
         \frac{p}{q} \cdot_{ {\mathrm{Quot}} } \frac{r}{s} \ := \ \frac{p * r}{q * s } \ \ \text{ in } \ 
        {\mathrm{Quot}} \ , \ \ \text{ while in } \ {\mathbf{R}( \mathbb H)} \text{ we multiply } \
         (x, y \text{ as above }):   
      \end{align*} 
      \begin{align}  \label{noch eine  wichtige Definition}  
           x * y   
          \ \ := \ \left[ \left\langle  \sum_{i=1}^{N} \sum_{j=1}^{K} \pm_{i,j,*} \ x_i \cdot y_j \, \right\rangle 
          \right]_{/ \cong} 
      \end{align}  
           where in the last term the signs \ $ \pm_{i,j,*} \in \{ +,- \}  $
           are determined by the rules \ $\pm_{i,j,*} :=  + $ if $\pm_{i,1}\pm_{j,2} = ++ $ or  
           $\pm_{i,1}\pm_{j,2} = -- $ \: and \:  $\pm_{i,j,*} :=  - $ if $\pm_{i,1}\pm_{j,2} = +- $ or 
           $\pm_{i,1}\pm_{j,2} = -+ $. \ The summation in the last term is made by a formal sum 
          of multiplications in the semigroup $ ( \mathbb H, \cdot) $, provided with signs \ $\pm_{i,j,*}$.  
    \begin{remark}       \label{erste Bemerkung}     \textrm
         The associativity of $({\mathbf{R}( \mathbb H)}, * )$ relies on the associativity 
         of the semigroup   $ ( \mathbb H, \cdot) $.  \ The construct $ ( {\mathbf{R}(\mathbb H)} ,  *)$
         is commutative if and only if the generating semigroup $ ( {\mathbb H} , \cdot )$  is commutative. 
    \end{remark}     
      For length $ N = 0 $ we have the empty sum \,   
      $ \left\langle \right\rangle :=  \left\langle  \sum_{i=1}^{0} \pm_{i,1} \, x_{i} \, \right\rangle $.
      For positive natural numbers   $ N, K $ we define the sum \   
   \begin{align}    \label{wichtige Definition}   
       x +_{\mathbb H} \, y \ := \  
       \left[ \left\langle  \sum_{l=1}^{\max\{N, K \}} \pm_{l,1} \, x_l \pm_{l,2} \, y_l \, \right\rangle 
        \right]_{/ \cong} 
       \  \ \text{ where}
    %
   \end{align}      
   \begin{align}  \label {N ungleich K}
      &   \text{ if} \quad K < N  \  \text{ we define \ } \  \pm_{l,1} \, x_l \pm_{l,2} \, y_l := 
         \pm_{l,1} \ x_{l} \, \ \, \text{ for } K < l \leq N ,  \\  
      &   \text{ if} \quad N < K  \   \text{ we define \ } \ \pm_{l,1} \, x_l \pm_{l,2} \, y_l := \, \pm_{l,2} \ y_l \ 
                                                      \       \text{ for } N < l \leq K . 
   \end{align} 
    Hence, with relation `$\approx_2$' the addition `$ +_{\mathbb H} $' in
    $ ({\mathbf{R}( \mathbb H)}, +_{\mathbb H}) $ is commutative. \
      For \ $ x = \left\langle  \sum_{i=1}^{N} \pm_{i} \, x_{i} \, \right\rangle $ \  we define \ 
      $ -x :=  \left\langle  \sum_{i=1}^{N} \mp_{i} \, x_{i} \, \right\rangle  $, \; where if \
      $ \pm_{i} \ \; \text{ is } $ `$ + $' \; $ \text{ then } \mp_{i} \ \text{ is } $ `$- $', $ \; 
      \text{ and vice versa}$. \ By the relation `$\approx_1$' 
      we get \ $ \left[ x \right]_{/ \cong} +_{\mathbb H} \left[-x\right]_{/ \cong} = 0_{\mathbf{R}(\mathbb H)} $. \\
         It holds for 
       $ x,y,z \in {\mathbf{R}( \mathbb H)}:  $                                    \\
         \centerline{$ \ (x +_{\mathbb H} \, y ) +_{\mathbb H} \, z = x +_{\mathbb H} \, (y +_{\mathbb H} \, z)$
          \ and \ $ (x * y) * z = x * (y * z) $.}       \\ 
          We have the associativity for both operations. The pair $ (\mathbf{R}{( \mathbb H)}, +_{\mathbb H} ) $ 
          is an Abelian group.                                                     \\
       We are defining two operations \ $ +_{\mathrm{Quot}}, \cdot_{\mathrm{Quot}} : \ 
         {\mathrm{Quot}} \times  {\mathrm{Quot}} \longrightarrow  {\mathrm{Quot}} . \ 
         \text{ For} \ p, q, x, y \in {\mathbf{R}( \mathbb H)} \ \text{ and } \ 
         q, y \neq  0_{{\mathbf{R}(\mathbb H)}} \ 
         \text{ we define the addition }  $ `$ +_{ {\mathrm{Quot}}}$'  $ \text{ in } \  {\mathrm{Quot}}  $, 
  \begin{align}    \label{Addition}       
    \frac{ p }{ q }  +_{ {\mathrm{Quot}}}  \frac { x }{ y } \: := \: 
    \frac{ p * y +_{\mathbb H} \, q * x }{ q * y } \ \ .  
  \end{align} 
   It holds for $ a,b,c \in  {\mathrm{Quot}} :  $                                                           \\
         \centerline{$ \ (a +_{ {\mathrm{Quot}}} b )  +_{ {\mathrm{Quot}}} c = a  +_{ {\mathrm{Quot}}} 
         (b +_{ {\mathrm{Quot}}} c)$ \ and \ $ (a  \cdot_{ {\mathrm{Quot}}} b) \cdot_{ {\mathrm{Quot}}} c 
                                      = a \cdot_{ {\mathrm{Quot}}} (b \cdot_{ {\mathrm{Quot}}} c) $. }       \\
       We have the associativity for both operations.                                                        \\
       In the case we can construct a field  $ \mathbf{F}{( \mathbb H)} $, the pair 
       $ (\mathbf{F}{( \mathbb H)},  +_{_F}) $ is an Abelian group, hence the semigroup  \ 
       $ ( \mathbb H,  \cdot ) $ \ and also the ring \, 
       $ ( {\mathbf{R}(\mathbb H)} , +_{\mathbb H}, *) $ \, have to be commutative.                         \\
       
   Furthermore we have inverse elements  and two neutral elements \ 
         $ 0_{\mathbf{R}(\mathbb H)} \text { and } 0_{ {\mathrm{Quot}}} $ \ in the pairs 
         $ ({\mathbf{R}(\mathbb H)},+_{\mathbb H}) \ \text { or } \ ( {\mathrm{Quot}},+_{ {\mathrm{Quot}}}), 
         \text { respectively, and we define the neutral element } 
         1_{ {\mathrm{Quot}}} $. 
         Note that we will omit the symbol `$ [ . ]_{/ \cong} $' 
         of equivalent classes.
    \begin{align}        
     &   x +_{\mathbb H} \ 0_{\mathbf{R}(\mathbb H)} = x  \ 
          \text { and } \ x +_{\mathbb H} \, -x =  0_{\mathbf{R}(\mathbb H)} \ , \                                \\
     &     1_{ {\mathrm{Quot}}} \cdot_{ {\mathrm{Quot}}}  (x/y)  \ := \ 
           (x/y) \cdot_{ {\mathrm{Quot}}}  1_{ {\mathrm{Quot}}} \ := \ x/y \ =: \ (x/y) \, / 1_{ {\mathrm{Quot}}} 
                                                                                          \ , \  \\
     &     (r/s) \cdot_{ {\mathrm{Quot}}}  (x/1_{ {\mathrm{Quot}}})  \ := \ \frac{r * x}{s}       
            \ =: \       (x/1_{ {\mathrm{Quot}}}) \, \cdot_{ {\mathrm{Quot}}} \ (r/s)     \ , \                   \\ 
     &     (x/y) \cdot_{ {\mathrm{Quot}}} 0_{ {\mathrm{Quot}}}  \ := \ 0_{ {\mathrm{Quot}}} \ =: \ 
            0_{ {\mathrm{Quot}}} \cdot_{ {\mathrm{Quot}}}  (x/y) 
    \end{align}                                 
       for all \ $ r,s,x,y \in {\mathbf{R}(\mathbb H)} $, \, for \, $ s,y \neq 0_{\mathbf{R}(\mathbb H)} $.
     We get that the pair  $ ({\mathbf{R}( \mathbb H)}, +_{\mathbb H}) $ is an Abelian group.   \\ 
   %
                                                  \\                     
    For the left distributivity in  \ $ ( {\mathbf{R}(\mathbb H)} , +_{\mathbb H}, *) $ \ we need to show 
      \begin{align}    \label{Links Distributivitaet} 
                   a * ( x +_{\mathbb H} \, y) = a * x +_{\mathbb H} \, a * y 
      \end{align}              
      We use the definition of the sum in line \eqref{wichtige Definition}. Additionally we define \
      $ a := \left\langle  \sum_{j=1}^{A} \pm_{j,3} \, a_{j} \, \right\rangle \ \in sums({\mathbb H)} $.
      It holds   
   \begin{align}
           a * ( x +_{\mathbb H} \,y)
     & \ = \ \left\langle  \sum_{j=1}^{A} \pm_{j,3} \, a_{j} \, \right\rangle \ * \ ( x +_{\mathbb H} \,y)      \\ 
     &  \ = \ \left\langle  \sum_{j=1}^{A} \pm_{j,3}  \, a_{j} \, \right\rangle \ * \ 
        \left\langle  \sum_{l=1}^{\max\{ N, K \}}  \pm_{l,1}  \, x_l \pm_{l,2}  \, y_l \, \right\rangle         \\
  %
     &     \ = \  \left\langle  \sum_{j=1}^{A}    
           \sum_{l=1}^{ \max\{ N, K \} } \pm_{j,l,4} \ a_j \, \cdot \, x_l  \  \pm_{j,l,5} \ 
                a_j \, \cdot \, y_l \ \right\rangle                                                         \\   
     &   \ = \ \left\langle  \sum_{j=1}^{A}  \sum_{l=1}^{ N } 
                          \pm_{j,l,4} \ a_j \, \cdot \, x_l  \  \right\rangle   \  +_{\mathbb H} \, \
         \ \left\langle  \sum_{j=1}^{A}  \sum_{l=1}^{ K }  \ \pm_{j,l,5} \  a_j \, \cdot \, y_l  \ 
                    \right\rangle   \ = \   a * x +_{\mathbb H} \, a * y                              
       \end{align}  
            The signs  $ \pm_{j,l,4} \in \{ +,- \}  $ \ are determined by the following rules: \ If \
            $ \pm_{j,3} \pm_{l,1} = + + $ or if $ \pm_{j,3} \pm_{l,1} =  -- $ we set $ \pm_{j,l,4} := + $. 
            In the case \: 
            $ \pm_{j,3} \pm_{l,1} = -+ $ \: or if \: $  \pm_{j,3} \pm_{l,1} = + - $  we define 
            $ \pm_{j,l,4} := - $.  \ 
            For $ \pm_{j,l,5} $ also a corresponding rule holds. The signs $ \pm_{j,l,5} \in \{ +,- \} $ \
            are conditioned by the pair of signs $ \pm_{j,3} \pm_{l,2} $.      \\
    %
            Please see the explanation after line \eqref{noch eine  wichtige Definition}. If \ 
            $ N \neq K  $ \ we use the rule \eqref {N ungleich K}  and the following one:
            If $ K < N $ we define \ $ \pm_{j,l,4} \, a_j \cdot x_l \pm_{j,l,5} \, a_j \cdot y_l \
            :=   \pm_{j,l,4} \, a_j \cdot x_l $ for $ K < l \leq N $, or if  $ K > N $ we define 
            $ \pm_{j,l,4} \, a_j \cdot x_l \pm_{j,l,5} \, a_j \cdot y_l := \pm_{j,l,5} \, a_j \cdot y_l $
            for $ N < l \leq K $.    \\
            This ensures the left distributivity of  $ ({\mathbf{R}( \mathbb H)}, +_{\mathbb H}, *) $.
            The right distributivity works in the same manner. 
                                         
      Therefore it holds in \,  $({\mathbf{R}( \mathbb H)}, +_{\mathbb H}, * )$ \ the laws of distributivity, \ i.e. 
          $({\mathbf{R}( \mathbb H)}, +_{\mathbb H}, * )$    \ is a ring.  
  \end{proof}        
      \begin{remark}          \textrm
          This ring $ \mathbf{R}( \mathbb H) $  is free of zero divisors and it has an unit if and only if the
          generating semigroup  $ ( \mathbb H, \cdot) $ has an unit
          $ 1_{ \mathbb H} $, i.e. it holds $ 1_{ \mathbb H} \cdot a = a = a \cdot 1_{ \mathbb H} $
          for all $ a \in \mathbb H $. The unit in the ring $ \mathbf{R}( \mathbb H) $ in this case is 
          $ \left[ \left\langle +1_{ \mathbb H} \right\rangle \right]_{/ \cong} $.    
      \end{remark}                                                  
   \begin{lemma}
            In the case that the semigroup $ (\mathbb H , \cdot) $ is commutative, we get that 
            both operations ` $+_{ {\mathrm{Quot}}} $' and ` $\cdot_{ {\mathrm{Quot}}}$' from the tripel \ 
            $ ( {\mathrm{Quot}}, +_{ {\mathrm{Quot}}}, \cdot_{ {\mathrm{Quot}}} ) $ \ are commutative.  
            Furthermore, from the equations $ x \cdot_{ {\mathrm{Quot}}} y = x \cdot_{ {\mathrm{Quot}}} z $ or 
            $ y \cdot_{ {\mathrm{Quot}}} x = z \cdot_{ {\mathrm{Quot}}} x $ the equation
            $ y = z $ follows, for $ x,y,z \in  {\mathrm{Quot}} $ and $ x\neq 0_{ {\mathrm{Quot}}} $. 
   \end{lemma} 
      \begin{proof}   
           The commutativity and the associativity of the semigroup $ (\mathbb H , \cdot) $ causes the
           commutativity and the associativity of both operations of the ring \ 
           $ ( {\mathbf{R}(\mathbb H)} , +_{\mathbb H}, *) $ \ and as a consequence  the commutativity and
           associativity of the operations in \ 
           $ ( {\mathrm{Quot}}, +_{ {\mathrm{Quot}}}, \cdot_{ {\mathrm{Quot}}} ) $, too. \\
          We prove the last claim of the lemma: \
           Let    $ x = p/q $ and $ y = r/s, z = v/w $ for $ p,q,r,s,v,w \in {\mathbf{R}( \mathbb H)} $ and   
           $ p,q,s,w \neq 0_{{\mathbf{R}}( \mathbb H)} $. We multiply the equation $ x \cdot_{_F} y = x \cdot_{_F} z $
           with the inverse $ q/p $ of $ x $  and we get $ y = z $. 
   \end{proof} 
    We introduce a third relation  `$\approx_3$' to cancel common factors in $ {\mathrm{Quot}} $ to ensure the
    distributivity. 
    With this relation and if the multiplication 
    `$ * $' is commutative, the ordinary relation $ a/b \approx c/d $ if and only if $ ad = bc $ is fulfilled.  \\
    Let $ r,x,y \in \mathbf{R}(\mathbb H), \ r,y \neq 0_{\mathbf{R}(\mathbb H)} $. We say that two elements
    have the relation  `$\approx_3$' , 
   \begin{align}
                \frac{x}{y}  \ \approx_3  \ \frac{ r * x }{ r * y }  \ \ .
   \end{align}
   Let `$ \cong3 $' be the smallest equivalence relation on $ {\mathrm{Quot}} $ that includes the relation
   `$\approx_3$' . 
   We name the equivalent classes by $[ . ]_{/ \cong3} $ \ . 
   \begin{definition}
        Let  $ {\mathbf{F}( \mathbb H)}  :=  {\mathrm{Quot}}_{/ \cong3} $.  \\
        Let  $ +_{_F}, \cdot_{_F} $ be the canonical  continuations on $ {\mathbf{F}( \mathbb H)} $ of the operations
        $ +_{ {\mathrm{Quot}}}, \cdot_{ {\mathrm{Quot}}} $  for equivalent classes. For instance, we define
        the addition `$ +_{_F} $'   by  
    \begin{align}    \label{nochmal die Addition}       
             \left[ \frac{ x }{ y } \right]_{/ \cong3}  +_{_F}  \left[ \frac { v }{ w } \right]_{/ \cong3}  \: := \: 
             \left[ \frac{ x * w +_{\mathbb H} \, y * v }{ y * w }  \right]_{/ \cong3} .  
    \end{align} 
   \end{definition}
    \begin{lemma}
            In the case of a commutative semigroup $ (\mathbb H , \cdot) $ the laws of distributivity hold in \
            $ ({\mathbf{F}( \mathbb H)}, +_{_F}, \cdot_{_F} ) $.   
   \end{lemma} 
      \begin{proof}   
           The distributivity of \ $ ({\mathbf{F}( \mathbb H)}, +_{_F}, \cdot_{_F} ) $ \  
           relies on the distributivity and commutativity of the ring \
            $({\mathbf{R}( \mathbb H)}, +_{\mathbb H}, * )$.   For instance, to prove the left distributivity in \
              $ ({\mathbf{F}( \mathbb H)}, +_{_F}, \cdot_{_F} ) $  
           we need to confirm the equation
     \begin{align}
          &   \frac{r * a }{s * a} \ \cdot_{_F}  \left( \frac {x * b}{y * b} +_{_F}  \frac {v *c}{w * c} \right) 
                               \ = \ 
              \left( \frac {r * a}{s * a} \ \cdot_{_F} \frac {x * b}{y *b} \right) +_{_F} 
              \left( \frac {r * a}{s * a} \ \cdot_{_F} \frac {v * c}{w * c} \right)  \ ,          \\
          &   a,b,c,r,s,x,y,v,w \in \mathbf{R}(\mathbb H), \quad a,b,c,s,y,w \notin 0_{\mathbf{R}(\mathbb H)} 
     \end{align}
          We omitted the symbols `$ [ . ]_{/ \cong3} $' .      
   \end{proof}  
    \begin{proposition}   \label{noch eine proposition}
             We can construct the set $ {\mathbf{F}( \mathbb H)} $, and if the generating semigroup 
             $ (\mathbb H , \cdot) $ is commutative, we have that  the tripel  \ 
             $ ({\mathbf{F}( \mathbb H)}, +_{_F}, \cdot_{_F} ) $ \ is a field, where  \\
             \centerline{$ +_{_F}, \cdot_{_F} :  {\mathbf{F}( \mathbb H)} \times  {\mathbf{F}( \mathbb H)}
                                                        \longrightarrow {\mathbf{F}( \mathbb H)} $.}                      \end{proposition}  
     \begin{proof}
           Most of the constructions are made already above. The operations are independent of the representatives of
           the equivalent classes. The division `$ \Big{/}_{_F} $' is defined by
               \begin{align}
                 \left[ \frac{r}{s} \right]_{/ \cong3} \ \Big{/}_{_F} \ \left[ \frac{h}{g} \right]_{/ \cong3} \ := \
                 \left[ \frac{r}{s}  \right]_{/ \cong3} \ \cdot_{_F} \ \left[ \frac{g}{h} \right]_{/ \cong3} \ := \
                 \left[ \frac{ r * g } { s * h } \right]_{/ \cong3}  
                \text{ for } s,g,h \neq 0_{\mathbf{R}(\mathbb H)} \ .
   \end{align}
     \end{proof}
      We use the symbols `$+_{\mathbb H}$' and  `$+_{_F}$'  for two different operations on 
      $ {\mathbf{R}( \mathbb H)} $ and $ {\mathbf{F}( \mathbb H)} $, respectively. The neutral elements in 
      $ ({\mathbf{F}( \mathbb H)}, +_{_F}) $ and  $ ({\mathbf{F}( \mathbb H)}, \cdot_{_F} ) $ \ are 
      $ 0_{{\mathbf{F}( \mathbb H)}} := [0_{ {\mathrm{Quot}}}]_{\cong3} $ and 
      $ 1_{{\mathbf{F}( \mathbb H)}} := [1_{ {\mathrm{Quot}}}]_{\cong3} $, respectivily. The inverse element of 
      $ [r/s]_{\cong3} $ in   $ {\mathbf{F}( \mathbb H)} $ is $  [s/r]_{\cong3} $, of course.   \\       
     We consider $\mathbb H$ as a part of $\mathbf{R}(\mathbb H)$  by the embedding $ e_1 $,
   \begin{align}
  &  e_1 : \mathbb H \hookrightarrow {\mathbf R}(\mathbb H), \ \ \  
         a \mapsto \  \left[ \left\langle +a \right\rangle  \right]_{/ \cong}  \ .   
   \end{align}  
     Note that $e_1: ( \mathbb H, \cdot) \rightarrow  ({\mathbf{R}}(\mathbb H), *) $ is a semigroup homomorphism,
     i.e. $ e_1 (a \cdot b)  = e_1(a) * e_1(b) $.   
      We constructed the set $ {\mathbf{F}( \mathbb H)} $, and in the special case of a commutative semigroup
      $ (\mathbb H , \cdot) $ we get that \, $({\mathbf{F}( \mathbb H)}, +_{_F}, \cdot_{_F} )$ \ is a field.   \\
      In this case of a commutative semigroup $ (\mathbb H , \cdot) $ 
      we regard $\mathbf{R}(\mathbb H)$  as a part of $\mathbf{F}(\mathbb H)$  by the embedding $ e_2 $,       \\  
  \centerline {$e_2:  \ ( {\mathbf{R}(\mathbb H)} , +_{\mathbb H}, *) \hookrightarrow \ 
               ({\mathbf{F}( \mathbb H)}, +_{_F} ,               \cdot_{_F} ) \ \     \text{ where }$ }        \\
    \ \ $ e_2( 0_{\mathbf{R}(\mathbb H)}) := 0_{ \mathbf{F}(\mathbb H)}, \ \text{ and for }  \
      \left[ \left\langle  \pm_1 x_1 \pm_2 x_2 \pm_3 x_3 \, \ldots \, \pm_{N-1}  x_{N-1} \pm_N x_N \,
        \right\rangle \right]_{/ \cong}  \in  \mathbf{R}(\mathbb H)  $ \ we define                             \\
    \begin{align}
        e_2 \left( \left[ \left\langle  \pm_1 x_1 \pm_2 x_2 \pm_3 x_3 \, \ldots \, \pm_{N-1}  x_{N-1} \pm_N x_N \,
        \right\rangle \right]_{/ \cong} \right) \, := \,  \frac{ \left[ \left\langle \pm_1 x_1 \pm_2 x_2 \pm_3 x_3 
        \, \ldots  \, \pm_{N-1}  x_{N-1} \pm_N x_N \, \right\rangle \right]_{/ \cong} \ } {1_{\mathbf{F}(\mathbb H)}}.
    \end{align}    
     We get that the map $ e_2 $ is a ring homomorphism, i.e. it holds \ 
     $ e_2 (x +_{\mathbb H} \, y)  = e_2(x) +_{_F} \,  e_2(y) $ \, and \,  
     $ e_2 (x * y)  = e_2(x) \cdot_{_F} e_2(y) $ \ for all $ x,y \in  {\mathbf{R}}(\mathbb H) $. Also it holds that \,
     $e_2  \circ  e_1: \ (\mathbb H , \cdot ) \longrightarrow ({\mathbf{F}( \mathbb H)}, \cdot_{_F} ) $ \, 
      is a semigroup homomorphism, i.e. \, $ e_2  \circ  e_1( a \cdot b) =
       e_2  \circ  e_1 (a) \cdot_{_F} e_2  \circ  e_1 (b) $, \ for \,  $ a,b \in {\mathbb H} $. 
   \begin{theorem}   \label{eindeutiger Ringhomomorphismus}
     Let \ $ f: ( \mathbb H, \cdot) \rightarrow (R, \cdot_{_{Ring}} ) $ \, be a semigroup homomorphism,   
     i.e. $ f(a \cdot b)= f(a) \, \cdot_{_{Ring}}  f(b)\, \text{ for } a,b \in {\mathbb H} $, where the tripel 
        $(R, +_{_{Ring}}, \cdot_{_{Ring}} ) $ is a ring.
                                                                  
      Then there exists a unique ring homomorphism $ f^{\sharp}:  ({\mathbf{R}(\mathbb H)} , +_{\mathbb H}, *) 
      \rightarrow  (R, +_{_{Ring}}, \cdot_{_{Ring}}) $ 
      such that the following diagram commutes, i.e. $ f =  f^{\sharp}  \circ e_1 $. 
   \end{theorem} 
   \begin{proof}    
        We call \, $ 0_{Ring} $  \, the zero element in $ (R,+_{_{Ring}}) $, i.e. $ r +_{_{Ring}} \ 0_{Ring} = r $, 
        for all $ r \in R $. 
        We abbreviate the element  $ x \in   sums ({\mathbb H}) $,  \, $ x \neq  0_{\mathbf{R}(\mathbb H)}  $
                                                                                                       \ by      \\ 
          \centerline{$ x := \left\langle \pm_{1} \, x_1 \pm_{2} \, x_2 \pm_{3} \, x_3 \, \ldots  \,
                       \pm_{N-1} \,  x_{N-1} \pm_{N} \, x_N \right\rangle $.}  \\                       
       We define the map \  $ f^{\sharp} $ \  by 
   \begin{align}  
      &   f^{\sharp} :  ( {\mathbf{R}(\mathbb H)} , +_{\mathbb H}, *) \longrightarrow  (R, +_{_{Ring}},
            \cdot_{_{Ring}}), \ \ \text{ we set} \ \  f^{\sharp}( 0_{\mathbf R({\mathbb H})}) := 0_{Ring} 
            \ \ \text{ and }       \\   
      &     f^{\sharp} \left( [x]_{/ \cong} \right)   \ :=  \
            \pm_{1} f(x_1) \pm_{2} f(x_2) \pm_{3} f(x_3) \, \ldots \, \pm_{N-1}  f(x_{N-1}) \pm_{N} f(x_N) \ . 
   \end{align}
   %
           Note that the addition on the right hand side is in the ring $ R $.            
      The uniqueness of \  $ f^{\sharp} $ \, is easy,  since 
         $ f (a) =  f^{\sharp}  \circ e_1 (a) = 
         f^{\sharp} \left( \left[ \left\langle +a \right\rangle  \right]_{/ \cong} \right) = + f(a) $ \ 
         for each \, $ a \in {\mathbb H}$. \ Therefore the values of $ f^{\sharp} $ are determined by the values of
         $ f $ on $ \mathbb H $. 
                                                                                       \\ 
                
         The map  \,  $ f^{\sharp} $ \, is a ring homomorphism, i.e. 
         $ f^{\sharp}(   x +_{\mathbb H}   y ) 
         = f^{\sharp}( x ) +_{Ring} f^{\sharp}( y ) $ \ and 
         $ f^{\sharp} \left( x * y \right) 
         = f^{\sharp} \left( x \right) \cdot_{Ring} \,  f^{\sharp} \left( y \right) $, 
                         where             \, $ x ,  y \in {\mathbf{R}(\mathbb H)} $.     
         The proof that $ f^{\sharp} $  truly is a ring homomorphism is straightforward. The map  $ f^{\sharp} $  
         is independent of representatives. \\
       \\ $$   
             \begin{xy}        \xymatrix{
             &   ( {\mathbb H}, \cdot )     \ar@{^{(}->}[rr]^{e_1}          \ar[drr]_{f}              
       &      &   ({\mathbf{R}(\mathbb H)} , +_{\mathbb H}, *)    \ar[d]_{f^{\sharp}}                         \\
       &      &                 &     (R, +_{_{Ring}}, \cdot_{_{Ring}})  &         }
             \end{xy}     
         $$
   \end{proof}  
   \begin{theorem}     \label{eindeutiger Koerperhomomorphismus}
      Let \ $ g: ( \mathbb H, \cdot) \rightarrow (F, \cdot_{_{Field}} ) $ \, be another semigroup homomorphism,   
      i.e. $ g(a \cdot b)= g(a) \, \cdot_{_{Field}}  g(b) $,
     where the tripel $(F, +_{_{Field}}, \cdot_{_{Field}} ) $ \, is a field. Since it is a field, it is a ring, too.  
                                                                                                
          By  the above Theorem \ref{eindeutiger Ringhomomorphismus} there exists a unique ring homomorphism 
          $ g^{\sharp} $ such that $ g = g^{\sharp} \circ e_1 $.  
       Additionally we assume that this map $ g^{\sharp} $ is  injective.   
                                                                 
      Then it exists a unique field homomorphism \
     $ g^{\nabla} : ({\mathbf{F}( \mathbb H)}, +_{_F}, \cdot_{_F} )  \longrightarrow (F, +_{_{Field}} ,                     \cdot_{_{Field}}) $ such that the diagram after next commutes, i.e. $ g =  g^{\nabla} \circ e_2 \circ e_1$ and  
     $ g^{\sharp} =  g^{\nabla} \circ e_2 $.
  \end{theorem}
  The following lemma is needed to define an embedding from the ring $ ({\mathbf{R}(\mathbb H)} , +_{\mathbb H}, *) $   into the field $ ({\mathbf{F}( \mathbb H)}, +_{_F} , \cdot_{_F}) $.
  Note that we just have built the field of fractions $ {\mathbf{F}( \mathbb H)} $ of the ring 
  $ {\mathbf{R}( \mathbb H)} $. 
     \begin{lemma}
        With the above conditions in this theorem 
        we get that the semigroup $ ( \mathbb H, \cdot) $ is commutative, and it follows that the operations in 
        the pairs $ ({\mathbf{R}(\mathbb H)} , *) $ and  $ ({\mathbf{F}(\mathbb H)} , \cdot_{_F} ) $ 
        are also commutative. As a consequence we are able to construct the field 
        $ ({\mathbf{F}( \mathbb H)}, +_{_F} , \cdot_{_F}) $.
     \end{lemma}   
     \begin{proof}
        We consider two elements $ g(a \cdot b) $ and $ g(b \cdot a) $ for $ a,b \in \mathbb H $. We have 
        $ g(b \cdot a)= g(b) \, \cdot_{_{Field}} g(a) = g(a) \, \cdot_{_{Field}} g(b) =  g(a \cdot b) $, since 
        $(F, +_{_{Field}}, \cdot_{_{Field}} ) $ \, is a field, i.e. $(F, \cdot_{_{Field}} ) $ is commutative. 
        Because of $ g = g^{\sharp} \circ e_1 $ and $ g^{\sharp} $ is injective it follows $ a \cdot b = b \cdot a $. 
        Please see Remark \eqref{erste Bemerkung}. The pair 
        $ ( {\mathbf{R}(\mathbb H)} , *)$  is associative and commutative.   This holds for 
        $ ( {\mathbf{F}(\mathbb H)} , \cdot_{_F} )$, too, due to the construction of
        $ {\mathbf{F}(\mathbb H)} $. Please see Proposition \eqref{noch eine proposition}. 
    \end{proof}  
   \begin{proof}  of Theorem \ref{eindeutiger Koerperhomomorphismus}.    \\         
         Since the tripel  $(F, +_{_{Field}}, \cdot_{_{Field}} ) $ is a field, it is also a ring. 
          By  Theorem \ref{eindeutiger Ringhomomorphismus} there exists a unique ring homomorphism $ g^{\sharp} $
          such that the following diagram commutes, i.e.   $ g = g^{\sharp} \circ e_1 $.
       \\ $$   
             \begin{xy}        \xymatrix{
             &   ( {\mathbb H}, \cdot )     \ar@{^{(}->}[rr]^{e_1}          \ar[drr]_{g}              
       &      &   ({\mathbf{R}(\mathbb H)} , +_{\mathbb H}, *)    \ar[d]_{g^{\sharp}}                       \\
       &      &                 &     (F, +_{_{Field}}, \cdot_{_{Field}})  &         }
             \end{xy}     
         $$    
           Let  $ 0_{_{Field}} ,  1_{_{Field}} $ 
           be the neutral elements in $ (F, +_{_{Field}} ) $ and $ (F, \cdot_{_{Field}} ) $, respectively.   \\
           The map $ g^{\sharp} $ is injective, i.e. the kernel of $ g^{\sharp} $ is  
           $ \{ 0_{_{Field}} \}  $.    \\ 
      \quad {\text {We define }}  \ \ $ g^{\nabla} : ({\mathbf{F}( \mathbb H)}, +_{_F}, \cdot_{_F} ) \longrightarrow
                                                   (F, +_{_{Field}} , \cdot_{_{Field}}) $.    \\  
       For arbitrary \
      $ p :=  \left\langle \pm_{1,1} \, p_1 \pm_{2,1} \, p_2 \pm_{3,1} \, p_3 \, \ldots \, 
                     \pm_{N-1,1} \, p_{N-1} \pm_{N,1} \, p_N \right\rangle $ \ \text{ and }   \\
      $ y :=  \left\langle \pm_{1,2} \, y_1 \pm_{2,2} \, y_2 \pm_{3,2} \, y_3 \, \ldots \,     
                \pm_{M-1,2} \, y_{M-1} \pm_{M,2} \, y_M \right\rangle , \ \ \ 
                p, y \in  sums ({\mathbb H}) , \ [ y ]_{/ \cong} \neq 0_{\mathbf{R}(\mathbb H)} $               \\
      \qquad {\text {we define }} 
    \begin{align}
      &     g^{\nabla} ( 0_{\mathbf{F}(\mathbb H)} ) :=  0_{_{Field}} \ \text{ and } \ 
            g^{\nabla} ( 1_{\mathbf{F}(\mathbb H)} ) :=  1_{_{Field}} \ \text{ and} \ 
            g^{\nabla} \left( \frac{[p]_{/ \cong}}{ 1_{\mathbf{F}(\mathbb H)} } \right) \ :=   
            g^{\sharp} \left( [p]_{/ \cong} \right) \ := \\ 
      &     \frac{\pm_{1,1} \, g(p_1) \pm_{2,1} \, g(p_2) \pm_{3,1} \, g(p_3) \, \ldots \, 
                                       \pm_{N-1,1} \, g(p_{N-1}) \pm_{N,1} \, g(p_N) }{ 1_{_{Field}} } \ ,       
     \end{align}                                   
           \text{ as well as for }  \ \ \  $ w :=  \frac{p}{y} $ 
       \ \      \text{ we define } 
    \begin{align}   
      &      g^{\nabla} \left( [w]_{/ \cong 3} \right) \, := \, 
                     \frac{ g^{\sharp}([p]_{/ \cong})} { g^{\sharp}([y]_{/ \cong})} \ = \
                     \frac{\pm_{1,1} \, g(p_1) \pm_{2,1} \, g(p_2) \pm_{3,1} \, g(p_3) \, \ldots \, 
                     \pm_{N-1,1} \, g(p_{N-1}) \pm_{N,1} \, g(p_N) }
                     {\pm_{1,2} \, g(y_1) \pm_{2,2} \, g(y_2) \pm_{3,2} \, g(y_3) \, \ldots \, 
                     \pm_{M-1,2} \, g(y_{M-1}) \, \pm_{M,2} \, g(y_M)} \ \ .  
    \end{align}   
      The condition that $ g^{\sharp} $ is injective ensures that the denominator $  g^{\sharp}([y]_{/ \cong}) $
      of \, $ g^{\nabla} ([w]_{/ \cong 3}) $ \, is not zero.  
      Note that if \ $ \frac{p}{y} =  \frac{r * a}{r * b} $, we have  
     \begin{align} 
      g^{\nabla} \left( \left[w\right]_{/ \cong 3} \right)  
      \ = \  \frac{ g^{\sharp}([p]_{ \cong}) }{ g^{\sharp}([y]_{ \cong})}  
      \ = \ \frac{  g^{\sharp}([r * a]_{ \cong})}{ g^{\sharp}( [r * b]_{ \cong})}    \ = \
         \frac{ g^{\sharp}([r]_{ \cong}) \cdot_{_{Field}} g^{\sharp}([a]_{ \cong})}
               { g^{\sharp}([r]_{ \cong}) \cdot_{_{Field}} g^{\sharp}([b]_{ \cong})} \ = \ 
         \frac{  g^{\sharp}([a]_{ \cong})}{  g^{\sharp}([b]_{ \cong})}  \ .
     \end{align}        
         The values of $ g^{\nabla} $ are determined uniquely by the values of
      $ g $ on $ \mathbb H $.  The uniqueness of \ $ g^{\nabla} $ \, is clear \ since
           $ g \left( a \right) =  g^{\nabla}  \circ e_2 \circ e_1(a) 
           = g^{\nabla} \circ e_2 \left( \left[ \left\langle +a \right\rangle  \right]_{/ \cong} \right)
           = g^{\nabla} \left( \frac{ \left[ \left\langle +a \right\rangle  \right]_{/ \cong} }
                                                            { 1_{\mathbf{F}(\mathbb H)} } \right)
           = + g \left( a \right) $, for $ a \in  \mathbb H $.  \\
           Please see the following commutative diagram. 
      \\ $$   
             \begin{xy}        \xymatrix{
             &    ( {\mathbb H}, \cdot )     \ar@{^{(}->}[rr]^{e_1}  \ar[drr]_{g}   
        &     & ({\mathbf{R}(\mathbb H)} , +_{\mathbb H}, *)  \ar@{^{(}->}[r]^{e_2}   \ar[d]_{g^{\sharp}} 
             &     ({\mathbf{F}( \mathbb H)}, +_{_F}, \cdot_{_F})    \ar[d]_{g^{\nabla}}        \\          
             &  &  &       (F, +_{_{Field}}, \cdot_{_{Field}})    \ar@{^{}=}[r]^{id}  
             &                 (F, +_{_{Field}}, \cdot_{_{Field}})           }
             \end{xy}     
         $$
  \end{proof}       
      { \textbf{ Acknowledgements}: }    We thank  Aida Rodriguez for a careful reading of the paper. 
    
           \newpage
           Author:   \\
               Dr. Volker Wilhelm Th\"urey       \\
               Hegelstrasse 101                  \\
               28201 Bremen, \   Germany         \\
               T: \  49 (0) 421 591777           \\
               E-Mail:  \   volker@thuerey.de 
    
\end{document}